\documentclass[12pt,reqno]{amsart}

\usepackage{latexsym}
\usepackage{amssymb}

\newcommand{\sgn}{\mbox{sgn}}

\newcommand{\R}{{\mathbb R}}

\newcommand{\Z}{{\mathbb Z}}

\newcommand{\half}{{\frac{1}{2}}}

\renewcommand{\phi}{\varphi}

\newcommand{\acal}{\mathcal{A}}

\newcommand{\ccal}{\mathcal{C}}

\newcommand{\hcal}{\mathcal{H}}

\newcommand{\ncal}{\mathcal{N}}

\newcommand{\tcal}{\mathcal{T}}

\newtheorem{theo}{{\sc Theorem}}[section]
\newtheorem{maintheo}{{\sc Theorem}}
\newtheorem{mainlem}{{\sc Lemma}}
\newtheorem{mainprop}{{\sc Proposition}}
\newtheorem{cor}[theo]{{\sc Corollary}}

\newtheorem{maincor}[maintheo]{{\sc Corollary}}
\newtheorem{lem}[theo]{{\sc Lemma}}

\title[Lower bounds on the Hausdorff measure of nodal sets  ]
{Lower bounds on the Hausdorff measure of nodal sets  }

\author{Christopher D. Sogge}
\author{Steve Zelditch}
\address{Department of Mathematics, Johns Hopkins University, Baltimore,
MD 21218, USA}
\address{Department of Mathematics, Northwestern University,
Evanston IL,  60208-2730, USA}

\thanks{Research partially supported by NSF grants, DMS-904839 and
  DMS-0904252}

\date{\today}

\begin{document}

\begin{abstract} Let
$\ncal_{\phi_{\lambda}}$ be the nodal hypersurface of a
$\Delta$-eigenfunction $\phi_{\lambda}$ of eigenvalue $\lambda^2$
on a smooth Riemannian manifold. We prove that
$\hcal^{n-1}(\ncal_{\phi_{\lambda}}) \geq C \lambda^{\frac74-\frac{3n}4} $ on
the surface measure of its nodal set. The best previous lower
bound was $e^{- C \lambda}$.
\end{abstract}

\maketitle

Let $(M, g)$ be a compact $C^{\infty}$ Riemannian manifold of
dimesion $n$, let $\phi_{\lambda}$ be an $L^2$-normalized
eigenfunction of the Laplacian,
$$\Delta \phi_{\lambda} = - \lambda^2 \phi_{\lambda},
$$ and let
$$\ncal_{\phi_{\lambda}} = \{x: \phi_{\lambda}(x) = 0\}$$ be its
nodal hypersurface. Let $\hcal^{n-1}(\ncal_{\phi_{\lambda}})$
denote its $(n-1)$-dimensional Riemannian hypersurface measure.
 In this note we prove \begin{maintheo} \label{LB} For any $C^{\infty}$ metric
 $g$, there exists a constant $C_g > 0$ so that
$$\hcal^{n-1}(\ncal_{\phi_{\lambda}}) \geq   C_g
\lambda^{\frac{7}{4} - \frac{3n}{4}}. $$
\end{maintheo}

The proof of Theorem \ref{LB} is based on a special case of a
general identity (Proposition \ref{BOUNDS}) which has other
interesting implications.  In \S \ref{IMPROVE} we consider
possible improvements that might be derived  from  other cases of
the identity in Proposition \ref{BOUNDS}.

Some background on  lower bounds on volumes of nodal
hypersurfaces: In \cite{Y}, S. T. Yau conjectured that for any
$C^{\infty}$ metric, one should have
\begin{equation} \label{DF} c \lambda \leq
\hcal^{n-1}(\ncal_{\phi_{\lambda}}) \leq C \lambda. \end{equation}
Here, and elsewhere in this article, $C, c$ denote some positive
constants depending only on $(M, g)$ and not on $\lambda$.  Both
the upper and lower bounds were proved for real analytic
$C^{\omega}$ metrics by Donnelly-Fefferman in \cite{DF}. However,
for $C^{\infty}$ metrics the best previous result appears to be
\begin{equation}\label{EXP}  C^{- \lambda} \leq
\hcal^{n-1}(\ncal_{\phi_{\lambda}}) \leq \lambda^{C \lambda}.
\end{equation} The upper bound was first proved in \cite{HS} and
the lower bound is proved in Theorem 6.2.5 of \cite{HL}. In
dimension 2, J. Br\"uning proved the lower bound of (\ref{DF})
(see also \cite{Sa}). Thus, our lower bound appears to be the
first one that breaks the exponential barrier in dimensions $n
\geq 3$. Perhaps surprisingly, the proof is quite simple.

In dimension 2,  Donnelly-Fefferman \cite{DF2} and  Dong \cite{D}
proved the upper bound $\hcal^1(\ncal_{\phi_{\lambda}}) \leq C
\lambda^{3/2}$ 
when $\dim M = 2$.  In dimensions $n \geq 3$,  the
Hardt-Simon  $\lambda^{\lambda}$ upper bound in (\ref{EXP}) still
seems to be the only known bound. The approach taken in this note
might lead to improvements in the upper bound, but not  as simply
as for the lower bound.

The proof of Theorem \ref{LB} is based on the following identity,
which was  inspired by a closely related identity  of  R. T. Dong
\cite{D} (see also  \cite{ACF}).

\begin{mainprop} \label{BOUNDS} For any $C^{\infty}$ Riemannian
manifold, we have,
\begin{equation} \label{f=1} 2 \int_{\ncal_{\phi_{\lambda}}} |\nabla
\phi_{\lambda}| dS =  \; \lambda^2 \int_M |\phi_{\lambda}| dV.
\end{equation}
More generally,  for any $f \in C^2(M)$,
\begin{equation} \label{DONGTYPE}  \int_M \left((\Delta + \lambda^2) f \right)  |\phi_{\lambda}| dV =
2 \; \int_{\ncal_{\phi_{\lambda}}}   f |\nabla \phi_{\lambda}| dS.
\end{equation}

\end{mainprop}

To obtain lower bounds on $\hcal^{n-1}(\ncal_{\phi_{\lambda}})$,
we need lower bounds on $||\phi_{\lambda}||_{L^1}$ and upper
bounds on $|\nabla \phi_{\lambda}|$. The following lower bound is
new:

\begin{mainprop} \label{CS}
For any $(M, g)$ and any $L^2$-normalized eigenfunction,
$||\phi_{\lambda}||_{L^1} \geq C_g \lambda^{- \frac{n-1}{4}}$.
\end{mainprop}

Combining Propositions \ref{BOUNDS} and \ref{CS} (and applying the
Schwartz inequality to $||\phi_{\lambda}||_{L^1}$), we obtain

\begin{maincor} \label{BOUNDSCOR} For any $C^{\infty}$ Riemannian
manifold,  there exists  constants $C, c> 0$ so that
$$ C \; \lambda^{2 - \frac{n-1}{4}}
 \leq  C \; \lambda^{2} ||\phi_{\lambda}||_{L^1} =  \int_{\ncal_{\phi_{\lambda}}}   |\nabla
\phi_{\lambda}| dS \leq c\; \lambda^2\; Vol(M)^{1/2}. $$
\end{maincor}
The upper bound is sharp (in terms of order of magnitude)  and is
achieved by the plane wave eigenfunctions on flat tori and by
highest weight spherical harmonics (see \S \ref{EXAMPLES}).

We then further use the local Weyl law  bound $||\nabla
\phi_{\lambda}||_{C^0} \leq C \lambda^{\frac{n + 1}{2}}$, again
valid on any $n$-dimensional Riemannian manifold, to obtain
\begin{mainlem} \label{BOUNDSCORa} For any $C^{\infty}$ Riemannian
manifold,  there exists  constants $C, c> 0$ so that
$$   \int_{\ncal_{\phi_{\lambda}}}   |\nabla
\phi_{\lambda}| dS  \leq \left(\sup_{x \in \ncal_{\phi_{\lambda}}}
|\nabla \phi_{\lambda}(x)|\right)
\hcal^{n-1}(\ncal_{\phi_{\lambda}}) \leq  C \lambda^{\frac{n +
1}{2}} \; \hcal^{n-1}(\ncal_{\phi_{\lambda}}).
$$
\end{mainlem}
Theorem \ref{LB} follows by combining Corollary \ref{BOUNDSCOR}
and  Lemma \ref{BOUNDSCORa} and by dividing both sides by
$\lambda^{\frac{n + 1}{2}}.$

In \S \ref{IMPROVE}, we discuss possible improvements of the lower
bound.  One of the inputs is  a lower bound
$||\phi_{\lambda}||_{L^1} \geq C_g\; \lambda^{- \frac{n-1}{4}}$,
on $L^1$ norms  of  $L^2$-normalized eigenfunctions, which is
valid for all smooth compact Riemannian manifolds.  Although this
is sharp, it is probably only achieved for very special Riemannian
manifolds.

The identity for general $f \in C^2(M)$ of Proposition
\ref{BOUNDS} can  be used to investigate the equidistribution of
nodal sets equipped with the surface measure $\lambda^{-2} |\nabla
\phi_{\lambda}| dS$. Various results of this kind are given in
\cite{Z2}.

 \subsection{Other level sets}

These results generalize easily to any level set
$\ncal_{\phi_{\lambda}}^c : = \{\phi_{\lambda} = c\}$. Let $\sgn
(x) = \frac{x}{|x|}$.

\begin{mainprop} \label{BOUNDSc} For any $C^{\infty}$ Riemannian
manifold, and any $f \in C(M)$ we have,

\begin{equation} \label{DONGTYPEc}  \int_M f (\Delta + \lambda^2)\; |\phi_{\lambda} - c| \;dV
+
\lambda^2 c \int f \mbox{\sgn} (\phi_{\lambda} - c) dV = 2\;
\int_{\ncal^c_{\phi_{\lambda}}}   f |\nabla \phi_{\lambda}| dS.
\end{equation}

\end{mainprop}

This identity has similar implications for
$\hcal^{n-1}(\ncal^c_{\phi_{\lambda}})$ and for the
equidistribution of level sets.  Note that if $c
> \sup |\phi_{\lambda}(x)|$ then indeed both sides are zero.

\begin{maincor}\label{cintro} For
$c \in {\mathbb R}$
$$\lambda^2\int_{\phi_\lambda\ge c}\phi_\lambda dV
=
\int_{\ncal^c_{\phi_{\lambda}}}   |\nabla \phi_{\lambda}| dS
\leq \lambda^2 Vol(M)^{1/2}. $$
Consequently, if $c>0$
$$\hcal^{n-1}(\ncal^c_{\phi_{\lambda}})
+\hcal^{n-1}(\ncal^{-c}_{\phi_{\lambda}})
\geq \;C_g \; \lambda^{2
- \frac{n + 1}{2}} \int_{|\phi_{\lambda}| \geq c} |\phi_{\lambda}|
dV.
$$

\end{maincor}

Of course,  $\int_{|\phi_{\lambda}| \geq c} \phi_{\lambda} dV \leq
||\phi_{\lambda}||_{L^1} $, so the lower bound for the  $c$-level
sets cannot be better than the  lower bound for nodal sets. Since
$\int_{|\phi_{\lambda}| \geq c} |\phi_{\lambda}| dV \geq
||\phi_{\lambda}||_{L^1} - c Vol\{(|\phi_{\lambda}| \geq c\}, $ in
the cases where $||\phi_{\lambda}||_{L^1} \geq \epsilon_0 > 0$,
this lower bound on
$\hcal^{n-1}(\ncal^c_{\phi_{\lambda}}\cup \ncal^{-c}_{\phi_{\lambda}})$
is
comparable to that for nodal sets when $c < \epsilon_0$.

We observe that by the co-area formula,
\begin{equation} \lambda^2 = \int_M |\nabla \phi_{\lambda}|^2 dV =
\int_{\min\{\phi_{\lambda}\}}^{\max \{\phi_{\lambda}\}}
\left(\int_{\ncal^c_{\phi_{\lambda}}} |\nabla \phi_{\lambda}| dS
\right) dc. \end{equation}  The bounds on the nodal set measure
the extent to which $0$ is an ``exceptional value" of
$\phi_{\lambda}$.

Finally, we should explain the connection of our results to those
of R. T. Dong.
 The identity (\ref{DONGTYPE}) is closely related  to the one in
 \cite{D}, Theorem 2.1:  Let $q_{\lambda}(x) = |\nabla
\phi_{\lambda}(x)|^2 + \frac{\lambda^2 \phi_{\lambda}^2}{n}$. Then
for any domain $\Omega \subset M$ with smooth boundary,
$$\hcal^{n-1} (\ncal_{\phi_{\lambda}} \cap \Omega) = \half \int_{\Omega} \frac{(\Delta + \lambda^2) |\phi_{\lambda}|}{
\sqrt{q_{\lambda}}} dV,
$$
where $dV$ is the volume form of $g$. More precisely, let
$\tcal_{\epsilon}$ denote the $\epsilon$-tube around the singular
set $\Sigma (\phi_{\lambda}) = \{x \in \ncal_{\phi_{\lambda}}:
\nabla \phi_{\lambda}(x) = 0\}, $ and define the integral by
$$\lim_{\epsilon \to 0} \half \int_{\Omega \backslash \tcal_{\epsilon}} \frac{(\Delta + \lambda^2) |\phi_{\lambda}|}{
\sqrt{q_{\lambda}}} dV,
$$
Dong's  formula shows that  $ \frac{(\Delta + \lambda^2)
|\phi_{\lambda}|}{ \sqrt{q_{\lambda}}} dV$ is the codimension one
Hausdorff measure  $\hcal^{n-1}$ on $\ncal_{\phi_{\lambda}}$. Dong
used this identity to obtain an upper bound on
$\hcal^1(\ncal_{\phi_{\lambda}})$ on surfaces.  We are  using a
simpler version where one does not divide by $q_{\lambda}$ to
prove a lower bound.

After the second author presented these results at Johns Hopkins,
we learned that similar results were obtained independently by
Colding and Minicozzi \cite{CM}.
 Using different methods they obtained the sharper lower bound of $\lambda^{\frac{3-n}4}$
  for $\hcal^{n-1}(\ncal_{\phi_{\lambda}})$. Just as we made the final revision
  to this article,  D. Mangoubi
 sent  the authors a preprint in which derives the lower bound
$\lambda^{3 - n - \frac{1}{n}}$. He also compares the methods of
the present article and of \cite{CM,M}.

We would like to thank D. Mangoubi, Q. Han and  W. Minicozzi
   for comments on earlier versions of the article. In particular, we thank W. Minicozzi
   for his helpful helpful comments on the exposition.

\section{Proof of Proposition \ref{BOUNDS}}

We first recall (see \cite{H, HHL, Ch, HHON})  that the the
singular set
$$\Sigma(\phi_{\lambda})= \{x \in \ncal_{\phi_{\lambda}}: \nabla
\phi_{\lambda}(x) = 0\} $$ satisfies $\hcal^{n-2}(\Sigma
(\phi_{\lambda})) < \infty$. Thus, outside of a codimension one
subset, $\ncal_{\phi_{\lambda}}$ is a smooth manifold, and the
Riemannian surface measure  $d S = \iota_{\frac{\nabla
\phi_{\lambda}}{|\nabla \phi_{\lambda}|}} dV_g$ on
$\ncal_{\phi_{\lambda}}$ is well-defined.

We note that the delta-function on $\ncal_{\phi_{\lambda}}$ is the
 Leray form $\delta(\phi_{\lambda}) =
\phi_{\lambda}^* \delta_0$,  i.e. the surface measure $\frac{d
Vol}{d \phi_{\lambda}} = \frac{dS}{|\nabla \phi_{\lambda}|}. $ The
measure $$\int f d\mu_{\lambda}: = \int_M f ( \Delta + \lambda^2)
|\phi_{\lambda}| dV$$  can thus be expressed as $|\nabla
\phi_{\lambda}| d S = |\nabla \phi_{\lambda}|^2
\delta(\phi_{\lambda})$.

Let us now give the proof of Proposition~\ref{BOUNDS}.  Clearly
the second identity implies the first one there, and so we need to
verify that we have
$$\int_M f d\mu_{\lambda} =  2 \int_{\ncal_{\phi_{\lambda}}}   f
|\nabla \phi_{\lambda}| dS.$$

%

We give two (slightly different) proofs.

\begin{proof}

 Since $d\mu_{\lambda}: = ( \Delta+ \lambda^2) |\phi_{\lambda}| dV = 0$
away from $\{\phi_{\lambda} = 0\}$ it is clear that this
distribution is supported on $\{\phi_{\lambda} = 0\}$. We let $f
\in C^{2}(M)$ and consider
$$\int_M f (\Delta + \lambda^2) |\phi_{\lambda}| dV =
\int_{|\phi_{\lambda}| \leq \delta} f (\Delta + \lambda^2)
|\phi_{\lambda}|  dV. $$ Almost all $\delta$ are regular values of
$\phi_{\lambda}$ by Sard's theorem and so we can apply Green's
theorem to such values, to obtain
$$ \int_{|\phi_{\lambda}| \leq \delta} f (\Delta + \lambda^2)
|\phi_{\lambda}|  dV - \int_{|\phi_{\lambda}| \leq \delta}
|\phi_{\lambda}|   (\Delta + \lambda^2) f  dV =
\int_{|\phi_{\lambda}| = \delta} (f \partial_{\nu}
|\phi_{\lambda}| - |\phi_{\lambda}|
\partial_{\nu} f) dS. $$
Here, $\nu$ is the outer unit normal and $\partial_{\nu}$ is the
associated directional derivative. For $\delta
> 0$, we have
\begin{equation} \label{NORMAL} \nu = \frac{\nabla \phi_{\lambda}}{|\nabla \phi_{\lambda}|}
\;\; \mbox{on}\;\; \{\phi_{\lambda} = \delta\}, \;\;\; \nu = -
\frac{\nabla \phi_{\lambda}}{|\nabla \phi_{\lambda}|} \;\;
\mbox{on}\;\; \{\phi_{\lambda} = - \delta\}. \end{equation}

 Letting $\delta \to 0$ (through the sequence of regular values)  we get
$$ \int_M f (\Delta + \lambda^2) |\phi_{\lambda}| dV = \lim_{\delta \to 0} \int_{|\phi_{\lambda}| \leq \delta} f (\Delta + \lambda^2)
|\phi_{\lambda}|  dV = \lim_{\delta \to 0}
 \int_{|\phi_{\lambda}|
= \delta}   f
\partial_{\nu} |\phi_{\lambda}|dS. $$
Since $|\phi_{\lambda}| = \pm \phi_{\lambda}$ on $\{\phi_{\lambda}
= \pm \delta\}$ and by (\ref{NORMAL}), we see that
$$ \begin{array}{lll} \int_M f (\Delta + \lambda^2) |\phi_{\lambda}| dV & = & \lim_{\delta \to 0}
 \int_{|\phi_{\lambda}|
= \delta}   f \frac{\nabla |\phi_{\lambda}| }{|\nabla
|\phi_{\lambda}| | } \cdot \nabla|\phi_{\lambda}|dS \\ && \\ & = &
\lim_{\delta \to 0} \sum_{\pm}
 \int_{\phi_{\lambda}
= \pm  \delta}   f |\nabla\phi_{\lambda}| dS \\ && \\& = & 2
\int_{\ncal_{\phi_{\lambda}}}   f |\nabla \phi_{\lambda}| dS.
\end{array}
$$

To justify the limit formula, we use the Gauss-Green formula of
geometric measure theory, \begin{equation} \label{GG}
\int_{\acal(0, \delta)} \mbox{div} F dy = - \int_{\partial^*
\acal(0, \delta)} F(y) \cdot \nu(y) d \hcal^{N-1}(y)
\end{equation} to the ``annulus" $\acal(0, \delta) = \{0 \leq
\phi_{\lambda} \leq \delta\}. $ Here, $\partial^* \acal_{0,
\delta}$ is the ``essential boundary" of $\acal_{0, \delta}$ (the
boundary in the sense of measure theory). In our case, the full
boundary is essential, and $\partial^* \acal(0, \delta) =
\ncal_{\phi_{\lambda}} \cup \ncal^{\delta}_{\phi_{\lambda}}. $
Also, $\nu$ is the unit normal, which in our case is $\nu = -
\frac{\nabla \phi_{\lambda}}{|\nabla \phi_{\lambda}|}. $ To obtain
the desired surface integral, we need to set $F = f \nabla
\phi_{\lambda}$, which is a smooth vector field. We obtain
$$\int_{\ncal_{\phi_{\lambda}}}  f |\nabla \phi_{\lambda}| d S  -
\int_{\ncal_{\phi_{\lambda}}^{\delta}} f |\nabla \phi_{\lambda}| d
S =  \int_{\acal(0, \delta)} (\nabla f \cdot \nabla
\phi_{\lambda}) - \lambda^2 \phi_{\lambda}) d V= O(\delta).
$$

The Gauss-Green formula (\ref{GG}) was proved by De Giorgi and
Federer under the assumption that $\hcal^{n-1}(\partial^* \acal(0,
\delta)) < \infty$, which holds for level sets of eigenfunctions
of $C^{\infty}$ metrics since $\hcal^{n-1}(\Sigma(\phi_{\lambda}))
= 0$ as $\hcal^{n-2}(\Sigma(\phi_{\lambda}))<\infty$.  We refer to Federer \cite{F} (Sect 2.10.6, page 173 and to
Theorem 4.5.11 p. 506.) We also refer to Theorem 1 on p. 209 of
\cite{EG} and \cite{P} for further discussion.

\end{proof}

We now give a second proof:

\begin{proof}

We first note that we can express $M$ as the disjoint union
$$M=\bigcup_{j=1}^{N_+(\lambda)}D^+_j \, \cup \, \bigcup_{k=1}^{N_-(\lambda)} D^-_k\, \cup
{\mathcal N}_\lambda,$$ where the $D^+_j$ and $D^-_k$ are the
positive and negative nodal domains of $\phi_\lambda$, i.e, the
connected components of the sets $\{\phi_\lambda>0\}$ and
$\{\phi_\lambda<0\}$.

Let us assume for the moment that $0$ is a regular value for
$\phi_\lambda$, i.e., $\Sigma=\emptyset$.  Then each $D^+_j$ has
smooth boundary $\partial D^+_j$, and so if $\partial_\nu$ is the
Riemann outward normal derivative on this set, by the Gauss-Green
formula we have
\begin{align*}
\int_{D^+_j}((\Delta+\lambda^2) f) \, |\phi_\lambda|\, dV
&=\int_{D^+_j}((\Delta+\lambda^2) f) \, \phi_\lambda \, dV
\\
&=\int_{D^+_j}f\, (\Delta+\lambda^2)\phi_\lambda dV-
\int_{\partial D^+_j}f \, \partial_\nu \phi_\lambda \, dS
\\
&=\int_{\partial D^+_j} f \, |\nabla \phi_\lambda|\, dS,
\end{align*}
using in the last step that $\phi_\lambda$ has eigenvalue
$\lambda^2$, and that $-\partial_\nu \phi_\lambda = |\nabla
\phi_\lambda|$ since $\phi_\lambda=0$ on $\partial D_j^+$ and
$\phi_\lambda$ decreases as it crosses $\partial D_j^+$ from
$D^+_j$. A similar argument shows that
\begin{equation*}
\int_{D_k^-} ((\Delta+\lambda^2)f )\, |\phi_\lambda|\, dV
=-\int_{D_k^-} ((\Delta+\lambda^2)f) \, \phi_\lambda \, dV =\int
f\,
\partial_\nu \phi_\lambda \, dS =\int_{\partial D^-_k} f\, |\nabla
\phi_\lambda| \, dS,
\end{equation*}
using in the last step that $\phi_\lambda$ increases as it crosses
$\partial D^-_k$ from $D^-_k$.  If we sum these two identities
over $j$ and $k$, we get
\begin{multline*}
\int_M ((\Delta+\lambda^2)f) \, |\phi_\lambda|\, dV
=\sum_j\int_{D^+_j} ((\Delta+\lambda^2)f) \, |\phi_\lambda|\, dV
+\sum_{k} \int_{D^-_k} ((\Delta+\lambda^2)f)\, |\phi_\lambda|\, dV
\\
=\sum_j \int_{\partial D_j^+} f\, |\nabla \phi_\lambda|\, dS +
\sum_k \int_{\partial D_k^-} f\, |\nabla \phi_\lambda|\, dS =
2\int_{{\mathcal N}_\lambda} f\, |\nabla \phi_\lambda|\, dS,
\end{multline*}
using the fact that ${\mathcal N}_\lambda$ is the disjoint union
of the $\partial D^+_j$ and the disjoint union of the $\partial
D^-_k$.

If $0$ is not a regular value of $\phi_\lambda$ we  use  the
Gauss-Green formula for domains with rough boundaries.  The
preceding argument yields
$$\int_{D_j^+} ((\Delta+\lambda^2)f) \, |\nabla\phi_\lambda|\, dV
=\int_{\partial D_j^+ } f|\nabla \phi_\lambda|\, dS,$$ and a
similar identity for the negative nodal domains. Since ${\mathcal
N_\lambda}\backslash \Sigma$ is the disjoint union of each of the
$\partial D_j^+ $ and the $\partial D_k^-$, we conclude that the
same equation holds even when  $0$ is not a regular value of
$\phi_\lambda$.
\end{proof}

The attractive feature of this  formula is that it immediately
gives rather strong results on the measure with respect to
$|\nabla \phi_{\lambda}| dS$ of the nodal set.


\begin{lem} \label{COR} For $f \in C^2(M)$, we have  $2 \int_{\ncal_{\phi_{\lambda}}} f  |\nabla \phi_{\lambda}| dS =  \lambda^2 \int_M f |\phi_{\lambda}| dV
+ O(1), $ \end{lem}

\begin{proof}  Our main identity \eqref{DONGTYPE}  gives, for any test function $f$,
$$2 \int_{\ncal_{\phi_{\lambda}}} f  |\nabla \phi_{\lambda}| dS = \int_M |\phi_{\lambda}|
(\lambda^2 +  \Delta) f dV =  \lambda^2 \int f |\phi_{\lambda}| dV
+ O(1),
$$
by the Schwartz inequality and the fact that
$||\phi_{\lambda}||_{L^2} = 1$.
\end{proof}

\subsection{Proof of Proposition \ref{CS}}

We now prove   Proposition \ref{CS}:
\begin{proof}

Fix a function $\rho\in {\mathcal S}({\mathbb R})$ having the properties
that $\rho(0)=1$ and $\hat \rho(t)=0$ if $t\notin [\delta/2,\delta]$, where
$\delta>0$ is smaller than the injectivity radius of $(M,g)$.
If we then set
$$T_\lambda f = \rho(\sqrt{-\Delta}-\lambda)f,$$
we have that $T_\lambda \phi_\lambda=\phi_\lambda$.  Also,
by Lemma 5.1.3 in \cite{So2}, $T_\lambda$ is an oscillatory
integral operator of the form
$$T_{\lambda} f(x) = \lambda^{\frac{n-1}{2}} \int_M e^{i \lambda
r(x,y)} a_\lambda(x, y) f(y) dy,$$
with $|\partial^\alpha_{x,y}a_\lambda(x,y)|\le C_\alpha$.
Consequently,
$||T_{\lambda}
\phi_{\lambda}||_{L^\infty} \le C \lambda^{\frac{n-1}{2}}
||\phi_{\lambda}||_{L^1}$, with $C$ independent of $\lambda$,
and so
$$1 = ||\phi_{\lambda}||_{L^2}^2 = \langle T \phi_{\lambda},
\phi_{\lambda} \rangle \leq ||T \phi_{\lambda}||_{L^\infty}
||\phi_{\lambda}||_{L^1} \leq  C \lambda^{\frac{n-1}{2}}
||\phi_{\lambda}||_{L^1}^2. $$

We can give another proof based on eigenfunction
estimates in \cite{So3}, which say that
$$\|\phi_\lambda\|_{L^p}\le C \lambda^{\frac{(n-1)(p-2)}{4p}},
\quad 2<p\le \tfrac{2(n+1)}{n-1}.
$$
If we pick such a $2<p<\tfrac{2(n+1)}{n-1}$, then by H\"older's inequality, we have
$$1=\|\phi_\lambda\|_{L^2}^{1/\theta}\le \|\phi_\lambda\|_{L^1} \, \|\phi_\lambda\|_{L^p}^{\frac1\theta-1}\le  \|\phi_\lambda\|_{L^1}\bigl(\, C\lambda^{\frac{(n-1)(p-2)}{4p}}\, \bigr)^{\frac1\theta-1},
\quad \theta=\tfrac{p}{p-1}(\tfrac12-\tfrac1p)=\tfrac{(p-2)}{2(p-1)},$$
which  implies $\|\phi_\lambda\|_{L^1}\ge c\lambda^{-\frac{n-1}4}$, since
$(1-\tfrac1\theta) \tfrac{(n-1)(p-2)}{4p}=\tfrac{n-1}4$.
\end{proof}

We remark that this lowerbound for $\|\phi_\lambda\|_{L^1}$ is sharp on the standard sphere, since
$L^2$-normalized highest weight spherical harmonics of degree $k$ with eigenvalue
$\lambda^2 = k(k+n-1)$ have $L^1$-norms which are bounded above
and below by $k^{(n-1)/4}$ as $k\to \infty$.  Similarly, the $L^p$-upperbounds that we used in the
second proof of this $L^1$-lowerbound is also sharp because of these functions.

\bigskip

\subsection{Lower bounds on $\hcal^{n-1}(\ncal_{\phi_{\lambda}})$}

We now  complete the proof of Theorem \ref{LB} along the lines
sketched in the introduction:
\medskip

\begin{proof} Corollary \ref{BOUNDSCOR} follows from Propositions \ref{BOUNDS} and  \ref{CS}.  It thus remains to prove Lemma \ref{BOUNDSCORa}.

We require the following standard bounds
\begin{equation} \label{GRADUB}  ||\nabla \phi_{\lambda}||_{C^0} \leq C\; \lambda^{1 +
\frac{n-1}{2}}.  \end{equation} The proof  follows from the local
Weyl law,
$$\sum_{j: \lambda_j \leq \lambda} |\nabla \phi_{\lambda}(x)|^2 =
C_n \lambda^{n + 2} + R(x, \lambda), $$ where $R(x, \lambda) =
O(\lambda^{n + 1}). $ See, e.g., Proposition 2.3 of \cite{Z}. It
follows that $ |\nabla \phi_{\lambda}(x)|^2  \leq |R(x, \lambda)|
= O(\lambda^{n + 1})$. Hence,
$$ \int_{\ncal_{\phi_{\lambda}}}   |\nabla \phi_{\lambda}| dS \leq
\lambda^{\frac{n+1}{2}}  \int_{\ncal_{\phi_{\lambda}}}
 dS. $$

We then divide by $\lambda^{\frac{n+1}{2}} $, use the identity of
Proposition \ref{BOUNDS}  and then use the lower bound of
Proposition \ref{CS} to complete the proof of Theorem \ref{LB}.

\end{proof}

 \subsection{General level sets} The proof of Proposition \ref{BOUNDSc} for general level sets is
 similar to that for $c = 0$, so we will be brief.

\begin{lem} \label{DISTRIBUTIONa} Suppose that $\{\phi_{\lambda} = c\}$ is a level set.
Then (\ref{DONGTYPEc}) holds, i.e. we have $$(\Delta + \lambda^2)
|\phi_{\lambda} - c| dV
+
\lambda^2\; c\; \sgn(\phi_{\lambda} - c)
= |\nabla \phi_{\lambda}|^2 \delta (\phi_{\lambda} - c).$$

\end{lem}

\begin{proof} When $\phi_{\lambda}(x) \not= c,$ we have
$$ (\Delta + \lambda^2) |\phi_{\lambda} - c| dV =
-
\sgn(\phi_{\lambda} - c) \lambda^2 c dV. $$ Hence the difference
of the two sides is supported on $\ncal^c_{\phi_{\lambda}}$. We
then repeat the calculation in Proposition \ref{BOUNDS} with the
sets $|\phi_{\lambda} - c | \leq \delta$, to get
$$\int_M f (\Delta + \lambda^2) |\phi_{\lambda} - c| dV =
-
\int_{|\phi_{\lambda} - c| \geq \delta} f
\left(\sgn(\phi_{\lambda} - c) \lambda^2 c\right) dV +
\int_{|\phi_{\lambda}- c| \leq \delta} f (\Delta + \lambda^2)
|\phi_{\lambda} - c|  dV. $$ By Green's theorem,
$$ \int_{|\phi_{\lambda} - c| \leq \delta} f (\Delta + \lambda^2)
|\phi_{\lambda} - c|  dV - \int_{|\phi_{\lambda} - c| \leq \delta}
|\phi_{\lambda} - c|  (\Delta + \lambda^2) f  dV =
\int_{|\phi_{\lambda} - c| = \delta} (f \partial_{\nu}
|\phi_{\lambda} - c| - |\phi_{\lambda} - c|
\partial_{\nu} f) dS. $$
 Letting $\delta \to 0$ we get
$$ \ \lim_{\delta \to 0} \int_{|\phi_{\lambda} - c| \leq \delta} f (\Delta + \lambda^2)
|\phi_{\lambda}- c|  dV = \lim_{\delta \to 0}
 \int_{|\phi_{\lambda} - c|
= \delta}   f
\partial_{\nu} |\phi_{\lambda} - c|dS. $$
We have
$$\partial_{\nu} = \frac{\nabla |\phi_{\lambda} - c| }{|\nabla
|\phi_{\lambda} - c| |} \cdot \nabla = \frac{\nabla
\phi_{\lambda}}{|\nabla \phi_{\lambda}|} \cdot \nabla,
\;\;(\mbox{on} \; \{|\phi_{\lambda} - c| = \delta\})$$ and as
before,
$$ \lim_{\delta \to 0}
 \int_{|\phi_{\lambda} - c|
= \delta}   f
\partial_{\nu} |\phi_{\lambda} - c|dS  =
2
 \int_{\ncal^c_{\phi_{\lambda}}}   f |\nabla \phi_{\lambda}| dS. $$

\end{proof}

 We now prove Corollary \ref{cintro}:

 \begin{proof} The first statement follows by integrating $\Delta$ by parts,
 and by using the identity,
 \begin{equation} \label{c} \begin{array}{lll} \int_M |\phi_{\lambda} - c| + c \;\sgn(\phi_{\lambda} - c)
 \;dV & = & \int_{\phi_{\lambda} > c} \phi_{\lambda} dV -
 \int_{\phi_{\lambda} < c} \phi_{\lambda} dV \\ && \\ &=& 2\int_{\phi_{\lambda} > c} \phi_{\lambda} dV ,
 \end{array} \end{equation}
 since $0=\int_M \phi_\lambda dV=\int_{\phi_\lambda>c}\phi_\lambda dV
 +\int_{\phi_\lambda<c}\phi_\lambda dV$.


Since for $c>0$ we have
 $\int_{\phi_\lambda>-c}\phi_\lambda dV =-\int_{\phi_\lambda< -c}\phi_\lambda
dV
=\int_{\phi_\lambda< -c}|\phi_\lambda|dV,$ we also have
$$\lambda^2\int_{|\phi_\lambda|\ge c} |\phi_\lambda| dV
=
\int_{\ncal^c_{\phi_{\lambda}}}   |\nabla \phi_{\lambda}| dS
+\int_{\ncal^{-c}_{\phi_{\lambda}}}   |\nabla \phi_{\lambda}| dS
,
\quad c>0,
$$
which yields the second part of the Corollary since $|\nabla \phi_\lambda|
=O(\lambda^{\frac{n+1}2})$.
\end{proof}

\subsection{A curious identity}

 If we set $f = \phi_{\lambda_k}$ in (\ref{DONGTYPE}),
we obtain

\begin{lem}\label{DIFFERENCE} $(\lambda_j^2 - \lambda_k^2) \int_M \phi_{\lambda_k} |\phi_{\lambda_j}| dV
= 2\int_{\ncal_{\phi_{\lambda_j}}} \phi_{\lambda_k} |\nabla
\phi_{\lambda_j}| dS. $ \end{lem}

\begin{cor}\label{CURIOUS}  Suppose that $\lambda_j$ is a multiple eigenvalue and
that $\lambda_k = \lambda_j$. Then
$$\int_{\ncal_{\phi_{\lambda_j}}} \phi_{\lambda_k} |\nabla
\phi_{\lambda_j}| dS = 0. $$
\end{cor}

For instance, on a circle we may consider the double eigenvalue
$-k^2$ with eigenfunctions $\cos k x, \sin k x$, where we may
verify the formula. Note that the nodal sets need not intersect.
More interesting examples include
 arithmetic flat tori and round spheres have
eigenvalues of high multiplicity (see \S \ref{EXAMPLES}).

By a
similar calculation we have, $$(\lambda_j^2 - \lambda_k^2) \int_M
|\phi_{\lambda_k}|\, |\phi_{\lambda_j}| \, dV =
2\int_{\ncal_{\phi_{\lambda_j}}} |\phi_{\lambda_k}| \, |\nabla
\phi_{\lambda_j}| \, dS - 2\int_{\ncal_{\phi_{\lambda_k}}}
|\phi_{\lambda_j}| \, |\nabla \phi_{\lambda_k}| \, dS, $$ hence if
$\lambda_j$ is a multiple eigenvalue and  $\lambda_k = \lambda_j$,
$$\int_{\ncal_{\phi_{\lambda_j}}} |\phi_{\lambda_k}|\, |\nabla
\phi_{\lambda_j}| \, dS = \int_{\ncal_{\phi_{\lambda_k}}}
|\phi_{\lambda_j}|\, |\nabla \phi_{\lambda_k}| \, dS. $$
The calculation follows from the fact, by Proposition~\ref{BOUNDS},
$(\Delta+\lambda_j^2)|\phi_{\lambda_j}|$ is the measure
$2|\nabla \phi_{\lambda_j}| \, dS$ supported on $\ncal_{\phi_{\lambda_j}}$.

\section{\label{EXAMPLES} Examples}

\noindent{\bf (i) Flat tori}
\medskip

We first consider the eigenfunctions $\phi_{k}(x) = \sin \langle
k, x \rangle$ ($k \in \Z^n$) on the flat torus ${\bf T} =
\R^n/\Z^n$. The zero set consists of the hyperplanes $\langle k, x
\rangle = 0 $ mod $2 \pi$. Also, $|\nabla \sin \langle k, x
\rangle|^2 = \cos^2 \langle k, x \rangle |k|^2$. Since $\cos
\langle k, x \rangle = 1$ when $\sin \langle k, x \rangle  = 0$
the integral is simply $|k|$ times the surface volume of the nodal
set, which is known to be of size $|k|$. So the upper bound of
Corollary \ref{BOUNDSCOR} is achieved in this example. Also, we
have $\int_{{\bf T}} |\sin \langle k, x \rangle| dx \geq C$. Thus,
our method gives  the sharp lower bound
$\hcal^{n-1}(\ncal_{\phi_{\lambda}}) \geq C \lambda^{1}$ in this
example.
\medskip

\noindent{\bf (ii) Zonal spherical harmonics on $S^2$}
\medskip

The spectral decomposition for the Laplacian is the orthogonal sum
of the spaces of spherical harmonics of degree $N$,
\begin{equation} L^2(S^2) = \bigoplus_{N=0}^{\infty} V_N,\;\;\;
\Delta |_{V_N} = \lambda_N Id.
\end{equation}
The eigenvalues   are   given by $\lambda_N^{S^2} = N (N + 1)$ and
the multiplicities are given by $m_N = 2 N + 1 $. A standard basis
is given by the (complex valued) spherical harmonics $Y^N_m$ which
transform by $e^{i m \theta}$ under rotations preserving the
poles.

We first consider  zonal spherical harmonics $Y^N_0 $ on $S^2$,
which are real-valued and maximize sup norms among $L^2$
normalized spherical harmonics. It is well known that $Y^N_0(r) =
\sqrt{\frac{(2 N + 1)}{2 \pi}} P_N(\cos r)$, where $P_N$ is the
$N$th Legendre function and the normalizing constant is chosen so
that $||Y^N_0||_{L^2(S^2)} = 1$, i.e.  $4 \pi \int_0^{\pi/2}
|P_N(\cos r)|^2 dv(r) = 1, $ where $dv(r) = \sin r dr$ is the
polar part of the area form. Its $L^1$ norm can be derived from
the asymptotics of Legendre polynomials in Theorem 8.21.2 of
\cite{S},
$$P_N(\cos \theta) = \sqrt{2} (\pi N \sin \theta)^{-\half} \cos
\left( (N + \half) \theta - \frac{\pi}{4} \right) + O(N^{-3/2}) $$
where the remainder is uniform on any interval $\epsilon < \theta
< \pi - \epsilon$. We have
$$||Y^N_0||_{L^1} = 4 \pi  \sqrt{\frac{(2 N +
1)}{2 \pi}} \int_0^{\pi/2} |P_N(\cos r)| dv(r) \sim C_0 > 0,$$
i.e. the $L^1$ norm is asymptotically a positive constant. Hence
$\int_{\ncal_{Y^N_0}}  |\nabla Y^N_0| ds \simeq C_0 N^2 $. In this
example $|\nabla Y_0^N|_{L^{\infty}} = N^{\frac{3}{2}}$ saturates
the sup norm bound. So the estimate of Lemma \ref{BOUNDSCORa}
produces the lower bound $\hcal^{n-1}(\ncal_{\phi_{\lambda}}) \geq
\lambda^{\half}$. The accurate lower bound is $\lambda$, as one
sees from the rotational invariance and by the fact that $P_N$ has
$N$ zeros. The defect in the argument is that the bound  $|\nabla
Y_0^N|_{L^{\infty}} = N^{\frac{3}{2}}$ is only obtained on the
nodal components near the poles, where each component has  length
$\simeq \frac{1}{N}$.
\bigskip

\noindent{\bf Gaussian beams} \medskip

A third example  is that  of real or imaginary parts of  highest
weight spherical harmonics $Y^N_N$  or other Gaussian beams along
a closed geodesic $\gamma$ (such as exist on equators of convex
surfaces of revolution). We refer to \cite{R} for background.
Gaussian beams are Gaussian shaped lumps which are concentrated on
$\lambda^{-\half}$ tubes $\tcal_{\lambda^{- \half}}(\gamma)$
around closed geodesics and have height $\lambda^{\frac{n-1}{4}}$.
We note that their $L^1$ norms decrease like
$\lambda^{-\frac{(n-1)}{4}}$, i.e. they saturate Proposition
\ref{CS}.  In such cases we have $\int_{\ncal_{\phi_{\lambda}}}
|\nabla \phi_{\lambda}| dS \simeq \lambda^2
||\phi_{\lambda}||_{L^1} \simeq \lambda^{2 - \frac{n-1}{4}}. $ It
is likely that Gaussian beams are minimizers of the $L^1$ norm
among $L^2$-normalized eigenfunctions of Riemannian manifolds.
Also, the gradient bound $||\nabla \phi_{\lambda}||_{L^{\infty}} =
O(\lambda^{\frac{n + 1}{2}})$ is far off for Gaussian beams, the
correct upper bound being $\lambda^{1 + \frac{n-1}{4}}$.  If we
use these estimates on $||\phi_{\lambda}||_{L^1}$ and $||\nabla
\phi_{\lambda}||_{L^{\infty}}$,  our method gives
$\hcal^{n-1}(\ncal_{\phi_{\lambda}}) \geq C \lambda^{1 -
\frac{n-1}{2} }$,  while $\lambda$ is the correct lower bound for
Gaussian beams in the case of surfaces of revolution (or any real
analytic case). The defect is again that the gradient estimate is
achieved only very close to the closed geodesic of the Gaussian
beam. Outside of the tube  $\tcal_{\lambda^{- \half}}(\gamma)$ of
radius $\lambda^{- \half}$  around the geodesic, the Gaussian beam
and all of its derivatives decay like $e^{- \lambda d^2}$ where
$d$ is the distance to the geodesic. Hence
$\int_{\ncal_{\phi_{\lambda}}} |\nabla \phi_{\lambda}| dS \simeq
\int_{\ncal_{\phi_{\lambda}} \cap \tcal_{\lambda^{-
\half}}(\gamma)} |\nabla \phi_{\lambda}| dS. $ Applying  the
gradient bound for Gaussian beams  to the latter integral
 gives $\hcal^{n-1}(\ncal_{\phi_{\lambda}} \cap
\tcal_{\lambda^{- \half}}(\gamma)) \geq C \lambda^{1 -
\frac{n-1}{2}}$, which is sharp since the intersection
$\ncal_{\phi_{\lambda}} \cap \tcal_{\lambda^{- \half}}(\gamma)$
cuts across $\gamma$ in $\simeq \lambda$ equally spaced points (as
one sees from the Gaussian beam approximation).

\section{\label{IMPROVE} Potential improvements}

As mentioned above, Theorem \ref{LB} is based on the identity in
Proposition \ref{BOUNDS} together with a lower bound on
$||\phi_{\lambda}||_{L^1}$ and an upper bound on $|\nabla
\phi_{\lambda}|$. Potential improvements could come from modifying
any of these three inputs.  The weakest link is the sup norm
estimate on $|\nabla \phi_{\lambda}|$. As we have seen in
examples, it is rarely achieved anywhere on $M$, and even when it
is, it is only achieved on a small portion of
$\ncal_{\phi_{\lambda}}$.

The first potential improvement is  to  use test functions $f$
other than $f \equiv 1$ to generate further identities.  For
instance if $f = |\nabla \phi_{\lambda}|^2$, then we get the
identity
$$\int_M |\nabla \phi_{\lambda}|^2 (\Delta + \lambda^2)
|\phi_{\lambda}| dV = \int_{\ncal_{\phi_{\lambda}}} |\nabla
\phi_{\lambda}|^3 dS. $$ Bochner's identity can be used to
simplify the left side. In this way one may try to use $L^p$
estimates rather than   sup norm estimates of $|\nabla
\phi_{\lambda}|$.

The lower bound of Proposition \ref{CS} is sharp among the class
of all $(M, g)$ since it is achieved by highest weight spherical
harmonics and other Gaussian beams. However, most $(M, g)$ do not
have Gaussian beams (which require existence of stable elliptic
closed geodesics), and one might hope to improve Proposition
\ref{CS} on manifolds with special geometries.
   It would  be particularly interesting to determine the $(M, g)$ or the  eigenfunction
sequences for which  $||\phi_{\lambda}||_{L^1} \geq C > 0 $. For
such eigenfunctions, we would have
\begin{equation} \label{SH}  (i) \;\;C \; \lambda^{2}
 \leq   \int_{\ncal_{\phi_{\lambda}}}   |\nabla
\phi_{\lambda}| dS \leq c\; \lambda^2\; Vol(M)^{1/2}, \;\;\; (ii)
\;\; \hcal^{n-1}(\ncal_{\phi_{\lambda}}) \geq C \lambda^{2 -
\frac{n + 1}{2}}.
\end{equation} Thus, in such  cases,  the upper and lower bounds of (i)  have the order of magnitude.
 As we have seen,  the usual exponential eigenfunctions on flat
tori or for zonal spherical harmonics on the standard sphere
satisfy these lower bounds.  It seems interesting to ask whether
$||\phi_{\lambda}||_{L^1} \geq C$ for eigenfunctions on negatively
curved manifolds (for instance). Indeed, in that case almost the
whole sequence of eigenfunctions satisfies $|\phi_{\lambda_j}|^2
\to 1$ (weak *) and $||\phi_{\lambda_j}||_{L^1} \to 0$ would
indicate (roughly speaking) that the eigenfunctions have a more
and more equidistributed set of high and narrow peaks separated by
low troughs. This need not contradict quantum ergodicity but it
may be a rare phenomenon.  We refer to \cite{Z2} for further
discussion.

  We note that minor improvements are also possible using the  argument of
  \cite{SoZ}, which
shows that the bound above on $||\nabla
\phi_{\lambda}||_{L^{\infty}}$ is rarely obtained anywhere for
metrics on $M$, and could be improved by a factor of
$\frac{1}{\log \lambda}$ if $(M, g)$ is negatively curved.

\subsection{Estimates on small balls}

Another potential source of improvements is to decompose $M$ into
small balls or cubes and to used scaled identities on each ball or
cube. More precisely, we could use test functions  $f =
\chi(\lambda (x-x_0))$ where $\chi \in C_0^{\infty}(\R^n)$ is a
smooth cutoff function, equal to one near $0$. We let
$D^{x_0}_{\lambda}$ denote the local dilation operator
$D^{x_0}_{\lambda} u(x_0 + y): = u(x_0 + \frac{y}{\lambda}) $ with
respect to some local coordinates.  It converts a high frequency
eigenfunction into a low frequency eigenfunction. Then
$\chi(\lambda (x-x_0)) = (D^{x_0}_{\lambda})^{-1} \chi(x - x_0)$,
and  $D_{\lambda}^{-1} \Delta D_{\lambda} \sim \lambda^2
\Delta^{x_0}_0$ (the Euclidean Laplacian with coordinates frozen
at $x_0$ in the local coordinates). Hence $\Delta \chi(\lambda
(x-x_0)) \simeq  \lambda^2 D_{\lambda}^{x_0} (\Delta^{x_0}_0
\chi)$ and from (\ref{DONGTYPE}) we have
 \begin{equation} \label{DONGTYPECUT} \begin{array}{lll}
\int_{\ncal_{\phi_{\lambda}} \cap B(x_0, \frac{C}{\lambda})}
\chi(\lambda (x-x_0)) |\nabla \phi_{\lambda}| dS & \simeq &
\lambda^2 \int_{B(x_0, \frac{C}{\lambda})} \left(
D_{\lambda}^{x_0} (I + \Delta^{x_0}_0 ) \chi
\right)|\phi_{\lambda}| dV.
\end{array}
\end{equation}

 One can produce similar inequalities on larger scaled balls. We
 can then cover the manifold with such  balls and partition  the  balls
 into the class of small balls  where $ |\nabla \phi_{\lambda}| $ is of average size
 $\lambda$ and those where it is of the much larger size $\lambda^{\frac{n+1}{2}}$ in our sup
 norm estimate. The latter balls are quite rare.   However, we leave this for future
 investigation.

\end{document}